\documentclass[11pt]{article}

\usepackage{amsfonts,amsmath,latexsym,color,epsfig,hyperref,mathtools}
\newtheorem{theorem}{Theorem}[section]
\newtheorem{lemma}{Lemma}[section]

\newtheorem{corollary}{Corollary}[section]

\newenvironment{proof}
      {\medskip\noindent{\bf Proof:}\hspace{1mm}}
      {\hfill$\Box$\medskip}
\def\qed{\ifvmode\mbox{ }\else\unskip\fi\hskip 1em plus 10fill$\Box$}

\input{epsf}

\makeatletter
\def\Ddots{\mathinner{\mkern1mu\raise\p@
\vbox{\kern7\p@\hbox{.}}\mkern2mu
\raise4\p@\hbox{.}\mkern2mu\raise7\p@\hbox{.}\mkern1mu}}
\makeatother

\def\ZZ{\mathbb{Z}}

\title{\vspace{-0.7cm}Hypergraph expanders from Cayley graphs}
\author{David Conlon\thanks{Mathematical Institute, Oxford OX2 6GG,
United Kingdom. Email: {\tt david.conlon@maths.ox.ac.uk}. Research
supported by a Royal Society University Research Fellowship and by ERC Starting Grant 676632.}}
\date{}

\begin{document}
\maketitle

\begin{abstract}
We present a simple mechanism, which can be randomised, for constructing sparse $3$-uniform hypergraphs with strong expansion properties. These hypergraphs are constructed using Cayley graphs over $\mathbb{Z}_2^t$ and have vertex degree which is polylogarithmic in the number of vertices. Their expansion properties, which are derived from the underlying Cayley graphs, include analogues of vertex and edge expansion in graphs, rapid mixing of the random walk on the edges of the skeleton graph, uniform distribution of edges on large vertex subsets and the geometric overlap property.
\end{abstract}

\section{Introduction}

Expander graphs are now ubiquitous in both mathematics and computer science. The problem of explicitly constructing these highly connected sparse graphs has drawn the attention of researchers from across both disciplines, who have uncovered deep and surprising connections to topics as diverse as Kazhdan's property (T) and the Ramanujan conjecture. Their usefulness has been known to computer scientists for some time, who have applied them to complexity theory, derandomisation, coding theory, cryptography and more, but they are now seeing increasing use in disparate areas of mathematics. We refer the interested reader to the excellent surveys~\cite{HLW06} and~\cite{L12} for further information. 

Given these successes, there has been a strong push in recent years towards defining and constructing high-dimensional, or hypergraph, expanders. There has already been a great deal of interesting work in this area (see, for example, the survey~\cite{L14}), but much more remains to be done. In particular, there are only a small number of examples known which satisfy the strongest notions of expansion. In the bounded-degree case, the standard examples are the Ramanujan complexes~\cite{LSV05}, defined in analogy to Ramanujan graphs~\cite{LPS88, M88} as finite quotients of certain affine buildings. These objects have remarkable properties, but can be somewhat difficult to analyse. The main result of this paper is a comparatively simple mechanism for constructing $3$-uniform expanders of low degree which satisfy many of the expansion properties discussed in the literature and which might prove interesting for further study. To say more, we first describe the mechanism.

Let $S$ be a subset of the finite abelian group $\mathbb{Z}_2^t$. We then let $H \coloneqq H(\ZZ_2^t, S)$ be the $3$-uniform hypergraph with vertex set $\mathbb{Z}_2^t$ and edge set consisting of all triples of the form $(x + s_1, x+ s_2, x+ s_3)$, where $x \in \mathbb{Z}_2^t$ and $s_1, s_2, s_3$ are distinct elements of $S$.  A useful alternative perspective on $H$ is to consider the Cayley graph Cay$(\mathbb{Z}_2^t, S')$, where $S'$ is the set $\{s_1 + s_2 : s_1, s_2 \in S, s_1 \neq s_2\}$. This is the graph with vertex set $\ZZ_2^t$ where $x, y \in \ZZ_2^t$ are joined if and only if $x + y \in S'$. Then $H$ is a $3$-uniform hypergraph on the same vertex set whose triples correspond to certain degenerate triangles in Cay$(\mathbb{Z}_2^t, S')$. Note that if $S$ contains no non-trivial solutions to the equation $s_1 + s_2 = s'_1 + s'_2$, then every vertex in $H$ is contained in exactly $3\binom{|S|}{3}$ edges and every pair of vertices is contained in either $0$ or $2|S| - 4$ edges.

We will see that this hypergraph inherits its expansion properties from the Cayley graph Cay$(\mathbb{Z}_2^t, S)$. However, when $|S| < t$, the Cayley graph Cay$(\mathbb{Z}_2^t, S)$ is not even connected, showing that $|S|$ will have to be at least logarithmic in the number of vertices. On the other hand, a celebrated result of Alon and Roichman~\cite{AR94} shows that logarithmic size will also suffice, in that if $S$ is a randomly chosen subset of $\ZZ_2^t$ of size $C t$, for $C$ sufficiently large, then Cay$(\mathbb{Z}_2^t, S)$ will, with high probability as $t \rightarrow \infty$, be an expander. The set $S$ may also be chosen explicitly, but the fact that a random choice works partially addresses a question raised repeatedly in the literature~\cite{EK17, P14, P17} as to whether there are random models for sparse high-dimensional expanders. 

Since random Cayley graphs over many other groups are known to have much better expansion properties (see~\cite{BGGT15} and its references), one might ask why we use $\ZZ_2^t$. To see why, we define, for each $x \in \ZZ_2^t$, the set $C_x := \{x + s : s \in S\}$. The hypergraph $H$ is then a union of $3$-uniform cliques, one on each set $C_x$. The key observation about $H$, and the reason why it serves as a hypergraph expander, is that each edge in Cay$(\mathbb{Z}_2^t, S')$ appears in at least two different sets of the form $C_x$. More specifically, if $(x, x + s_1 + s_2)$ is an edge of Cay$(\mathbb{Z}_2^t, S')$, then this edge is contained in both $C_{x+s_1}$ and $C_{x+s_2}$. This property, for which the fact that $\ZZ_2^t$ is abelian is crucial, means, for instance, that a random walk on the edges of $H$ never becomes trapped inside a particular set $C_x$.

To say more about the notions of expansion satisfied by our construction, we require some notation. Given a $3$-uniform hypergraph $H$, we let $E$ be the collection of pairs of distinct vertices $(u, v)$ for which there exists an edge $(u, v, w)$ of $H$ containing $u$ and $v$. The corresponding graph will be referred to as the {\it skeleton} of $H$. To distinguish the edges of $H$ from the set $E$, we will write $T$ for these edges and refer to them as the triples of $H$, reserving the term edges for the elements of $E$.

For any subset $F$ of the skeleton $E$ of a hypergraph $H$, we may define two notions of neighbourhood, the \emph{edge neighbourhood $N_E(F)$} and the \emph{triple neighbourhood $N_T(F)$}, by
\[N_E(F) = \{e \in E \setminus F: e \cup f \in T \textrm{ for some } f \in F\}\]
and, writing $\Delta(F)$ for the set of triangles in $F$,
\[N_T(F) = \{t \in T: t \supset f \textrm{ for some } f \in F \textrm{ and } t \notin \Delta(F)\}.\]
We then let
\[h_E(H) = \min_{\{F : |F| \leq |E|/2\}} \frac{|N_E(F)|}{|F|} \textrm{ and } h_T(H) = \min_{\{F : |F| \leq |E|/2\}} \frac{|N_T(F)|}{|F|}\]
and say that $H$ is an {\it $\epsilon$-edge-expander} if $h_E(H) \geq \epsilon$ and an {\it $\epsilon$-triple-expander} if $h_T(H) \geq \epsilon$. Note that these notions are the direct analogues of the standard notions of vertex and edge expansion in graphs. Our main result is that under suitable conditions on Cay$(\ZZ_2^t, S)$, the hypergraph $H(\ZZ_2^t, S)$ is an $\epsilon$-edge-expander and an $\epsilon |S|$-triple-expander for some $\epsilon > 0$. 

To state the result formally, recall that the eigenvalues of an $n$-vertex graph $G$ are the eigenvalues $\lambda_1 \geq \dots \geq \lambda_n$ of its adjacency matrix $A$. When $G$ is $d$-regular, $\lambda_1 = d$. One of the key facts about $d$-regular expander graphs is that if $\lambda(G) = \max_{i \neq 1} |\lambda_i|$ is bounded away from $d$, then the graph $G$ exhibits strong expansion properties. We show that this behaviour carries over to the derived hypergraph $H$.

\begin{theorem} \label{thm:main}
Suppose that $S \subseteq \ZZ_2^t$ contains no non-trivial solutions to the equation $s_1 + s_2 = s'_1 + s'_2$ and $\lambda(Cay(\ZZ_2^t, S)) \leq (1 - \epsilon)|S|$. Then the hypergraph $H(\ZZ_2^t, S)$ is an $\frac{\epsilon^2}{128}$-edge-expander and an $\frac{\epsilon^2}{64} |S|$-triple-expander.
\end{theorem}

By the Alon--Roichman theorem~\cite{AR94}, a random set $S \subseteq \ZZ_2^t$ of order $Ct$, for $C$ a sufficiently large constant, will a.a.s.~have the required properties with $\epsilon = 1/2$. Therefore, writing $n = |\ZZ_2^t|$, the theorem implies that there exists a hypergraph with $n$ vertices and $O(n \log^3 n)$ triples which is a $2^{-9}$-edge-expander and a $(2^{-8} \log n)$-triple-expander. 

Getting back to random walks, we follow Kaufman and Mass~\cite{KM16, KM162} in defining the {\it random walk} on a $3$-uniform hypergraph $H$ to be a sequence of edges $e_0, e_1, \dots \in E$ such that

\begin{enumerate}

\item
$e_0$ is chosen from some initial probability distribution $\mathbf{p}_0$ on $E$,

\item
for every $i \geq 1$, $e_i$ is chosen uniformly at random from the neighbours of $e_{i-1}$, that is, the set of $f \in E$ such that $e_{i-1} \cup f$ is an edge of $H$.

\end{enumerate}

We say that the random walk is {\it $\alpha$-rapidly mixing} if, for any initial probability distribution $\mathbf{p}_0$ and any $i \in \mathbb{N}$,
\[\|\mathbf{p}_i - \mathbf{u}\|_2 \leq \alpha^i,\]
where $\mathbf{p}_i$ is the probability distribution on $E$ after $i$ steps of the walk, $\mathbf{u}$ is the uniform distribution on $E$ and $\|\mathbf{x}\|_2 = (\sum_i x_i^2)^{1/2}$ for $\mathbf{x} = (x_1, x_2, \dots, x_n) \in \mathbb{R}^n$. As a corollary of Theorem~\ref{thm:main}, we can show that under appropriate conditions on Cay$(\ZZ_2^t, S)$ the hypergraph $H(\ZZ_2^t, S)$ is $\alpha$-rapidly mixing for some $\alpha < 1$. 
 
\begin{corollary} \label{cor:mixing}
Suppose that $S \subseteq \ZZ_2^t$ contains no non-trivial solutions to the equation $s_1 + s_2 = s'_1 + s'_2$ and $\lambda(Cay(\ZZ_2^t, S)) \leq (1 - \epsilon)|S|$.
Then the hypergraph $H(\ZZ_2^t, S)$ is $\alpha$-rapidly mixing with $\alpha = 1 - \Omega(\epsilon^4)$.
\end{corollary}

The hypergraph $H(\ZZ_2^t, S)$ also satisfies several other properties, one notable example being the following pseudorandomness condition.

\begin{theorem} \label{thm:pseudo}
Suppose that $S_t \subseteq \ZZ_2^t$ is a sequence of sets with $t \rightarrow \infty$, where $S_t$ contains no non-trivial solutions to the equation $s_1 + s_2 = s'_1 + s'_2$ and $\lambda(Cay(\ZZ_2^{t}, S_t)) = o(|S_t|)$. Then, for any sets $A, B, C \subseteq \ZZ_2^{t}$ with $|A| = \alpha |\ZZ_2^{t}|$, $|B| = \beta |\ZZ_2^{t}|$ and $|C| = \gamma |\ZZ_2^{t}|$, where $\alpha, \beta$ and $\gamma$ are fixed positive constants, the number of triples $T(A, B, C)$ in $H(\ZZ_2^{t}, S_t)$ with one vertex in each of $A, B$ and $C$ is $(1 + o(1)) \alpha \beta \gamma |S_t|^3 |\ZZ_2^{t}|$.
\end{theorem}

That is, provided $\lambda(Cay(\ZZ_2^{t}, S))$ is small, the density of edges between any three large vertex subsets $A$, $B$ and $C$ in $H(\ZZ_2^t, S)$ is asymptotic to the expected value, as it would be in a random hypergraph of the same density. 

As noted in~\cite{P17, PRT16}, any sequence of hypergraphs satisfying the conclusion of Theorem~\ref{thm:pseudo} will also satisfy Gromov's geometric overlap property~\cite{G10}. We say that a $3$-uniform hypergraph $H$ has the {\it $c$-geometric overlap property} if, for every embedding $\varphi: V(H) \rightarrow \mathbb{R}^2$ of the vertices of $H$ in the plane, there is a point $x \in \mathbb{R}^2$ which is contained in the convex hull of at least a $c$-proportion of the triples of $H$. A result of Boros and F\"uredi~\cite{BF84} (which was generalised to higher dimensions by B\'ar\'any~\cite{B82}) says that the family of complete $3$-uniform hypergraphs has the $c$-geometric overlap property with $c = \frac{2}{9} - o(1)$ and this constant is known to be sharp~\cite{BMN10}. Much more recently, answering a question of Gromov~\cite{G10}, Fox, Gromov, Lafforgue, Naor and Pach~\cite{FGLNP12} found a number of families of bounded-degree hypergraphs with the $c$-geometric overlap property for some fixed positive $c$. Indeed, one of their constructions~\cite[Section 4.1]{FGLNP12} is a close relative of ours and satisfies a version of Theorem~\ref{thm:pseudo}, but lacks the intersection property between the sets $C_x$ which is necessary to guarantee the properties encapsulated in Theorem~\ref{thm:main}. 

To see how the geometric overlap property follows from the conclusion of Theorem~\ref{thm:pseudo}, we recall Pach's selection theorem~\cite{P98}, which gives a stronger guarantee than the Boros--F\"uredi result, saying that for any set of $n$ points in the plane there exist three sets $A$, $B$ and $C$, each of order $\Omega(n)$, and a point $x \in \mathbb{R}^2$ such that every triangle $abc$ with $a \in A$, $b \in B$ and $c \in C$ has $x$ in its convex hull. Suppose now that $\varphi : V(H) \rightarrow \mathbb{R}^2$ is an embedding of the vertices of $H(\ZZ_2^t, S_t)$ in the plane and let $A$, $B$ and $C$ be the sets and $x$ the point guaranteed by applying Pach's theorem to this point set. By Theorem~\ref{thm:pseudo}, the number of triples of $H$ with one vertex in each of $A$, $B$ and $C$ is itself a positive proportion of the total number of triples in $H$. Since each one of these triples has $x$ in its convex hull, we have deduced the following corollary.

\begin{corollary} \label{cor:geom}
Suppose that $S_t \subseteq \ZZ_2^t$ is a sequence of sets with $t \rightarrow \infty$, where $S_t$ contains no non-trivial solutions to the equation $s_1 + s_2 = s'_1 + s'_2$ and $\lambda(Cay(\ZZ_2^{t}, S_t)) = o(|S_t|)$. Then the family of $3$-uniform hypergraphs $H(\ZZ_2^t, S_t)$ has the $(c - o(1))$-geometric overlap property for some $c > 0$.
\end{corollary}

As in~\cite{FGLNP12}, we can also recover the sharp constant $c = \frac{2}{9}$ by following Bukh's proof~\cite{B06} of the Boros--F\"uredi result. However, we omit the proof of this result, referring the reader instead to~\cite{FGLNP12}. 

It remains an open problem to determine whether an analogue of Corollary~\ref{cor:geom} holds for the stronger topological overlap property. We say that a $3$-uniform hypergraph $H$ has the {\it $c$-topological overlap property} if, for every continuous map $\varphi: X \rightarrow \mathbb{R}^2$ from the simplicial complex $X = (V, E, T)$ of the hypergraph $H$ to the plane, there is a point $x \in \mathbb{R}^2$ which is contained in the image of at least a $c$-proportion of the triples of $H$. Gromov~\cite{G10} generalised the Boros--F\"uredi result (and B\'ar\'any's result), showing that the family of complete $3$-uniform hypergraphs has the $c$-topological overlap property with $c = \frac{2}{9} - o(1)$. The difficult problem of constructing a family of $3$-uniform hypergraphs of bounded degree with the $c$-topological overlap property for some fixed positive $c$ was only solved recently by Kaufman, Kazhdan and Lubotzky~\cite{KKL16} and then extended to higher uniformities by Evra and Kaufman~\cite{EK17}. Their work relies on the properties of Ramanujan complexes, but we conjecture that our construction gives another simpler example, albeit one with polylogarithmic rather than constant degree.

\section{Proofs}

We begin with a simple lemma relating the eigenvalues of Cay$(\ZZ_2^t, S')$ to those of Cay$(\ZZ_2^t, S)$. In particular, this means that the skeleton Cay$(\ZZ_2^t, S')$ of our hypergraph $H(\ZZ_2^t, S)$ is an expander whenever Cay$(\ZZ_2^t, S)$ is. Here and throughout this section, we will use the shorthands $n = |\ZZ_2^t|$ and $d = |S|$.

\begin{lemma} \label{lem:eig}
Suppose that $S \subseteq \ZZ_2^t$ contains no non-trivial solutions to the equation $s_1 + s_2 = s'_1 + s'_2$ and that the eigenvalues of the Cayley graph Cay$(\ZZ_2^t, S)$ are $\lambda_i$ for $i = 1, \dots, n$. Then the eigenvalues of the Cayley graph Cay$(\ZZ_2^t, S')$, where $S' = \{s_1 + s_2 : s_1, s_2 \in S, s_1 \neq s_2\}$, are $\frac{1}{2}(\lambda_i^2 - d)$ for $i = 1, \dots, n$.
\end{lemma}

\begin{proof}
Let $A$ be the adjacency matrix of Cay$(\ZZ_2^t, S)$. It will suffice to show that the adjacency matrix of Cay$(\ZZ_2^t, S')$ is $\frac{1}{2}(A^2 - d I)$. To see this, note that $A^2_{xy}$ is the number of solutions to $x + y = s_1 + s_2$ with $s_1, s_2 \in S$. When $x \neq y$, the assumption that there are no non-trivial solutions to $s_1 + s_2 = s'_1 + s'_2$ tells us that $A^2_{xy} = 2$ or $0$ depending on whether or not $x + y$ are joined in Cay$(\ZZ_2^t, S')$ or not. When $x = y$, $A_{xx}^2 = d$, corresponding to the $d$ solutions $x + x = s + s = 0$ for all $s \in S$. The result follows. 
\end{proof}

Given two multisets $V$ and $W$ taken from the vertex set of a graph with edge set $E$, we write $e(V,W)$ to denote $\sum_{v \in V, w \in W} 1_E(vw)$. We will also use $v_x$ and $w_y$ to denote the multiplicity of $x$ in $V$ and $y$ in $W$, respectively. In what follows, we will need a slight variant of the expander mixing lemma which applies to multisets. Since the proof is identical to the usual expander mixing lemma, we omit it.

\begin{lemma}[Expander mixing lemma] \label{lem:eml}
Suppose that $G$ is an $(n, d, \lambda)$-graph, that is, $G$ has $n$ vertices of degree $d$ and all eigenvalues, save the largest, have absolute value at most $\lambda$. Then, for any two multisets $V, W \subseteq V(G)$,
\[\left|e(V, W) - \frac{d}{n} |V||W|\right| \leq \lambda \sqrt{\left(\sum_{x \in V} v_x^2 - \frac{|V|^2}{n}\right) \left(\sum_{y \in W} w_y^2 - \frac{|W|^2}{n}\right)}.\]
\end{lemma}

We are already in a position to show that $H(\ZZ_2^t, S)$ satisfies the pseudorandomness property encapsulated in Theorem~\ref{thm:pseudo}. This is a simple corollary of the following more precise result.

\begin{theorem} \label{lem:comb}
Suppose that $S \subseteq \ZZ_2^t$ contains no non-trivial solutions to the equation $s_1 + s_2 = s'_1 + s'_2$ and the eigenvalues of the Cayley graph Cay$(\ZZ_2^t, S)$ satisfy $|\lambda_i| \leq \lambda$ for all $i = 2, \dots, n$. Then, for any sets $A, B, C \subseteq \ZZ_2^t$ with $|A| = \alpha n$, $|B| = \beta n$ and $|C| = \gamma n$, the number of triples $e(A, B, C)$ in $H(\ZZ_2^t, S)$ with one vertex in each of $A, B$ and $C$ is
\[(d^3 - d^2) \alpha \beta \gamma n \pm 2 \mu d \sqrt{\alpha \beta} \gamma n \pm \lambda d^2 \sqrt{\alpha \beta \gamma} n \pm \lambda \sqrt{\mu} d (\alpha \beta)^{1/4} \sqrt{\gamma} n \pm 2 d^2 \alpha \beta n \pm 4 \mu \sqrt{\alpha \beta} n,\]
where $\mu = \frac{1}{2} (\lambda^2 + d)$.
\end{theorem}

\begin{proof}
Since Cay$(G, S')$ is $\binom{d}{2}$-regular and, by Lemma~\ref{lem:eig}, has all eigenvalues, except the largest, at most $\mu = \frac{1}{2}(\lambda^2 + d)$ in absolute value, the number of pairs in the skeleton of $H(\ZZ_2^t, S)$ that have one vertex in $A$ and one vertex in $B$ is, by the expander mixing lemma,
\begin{equation} \label{eqn:comexp1} 
\binom{d}{2} \alpha \beta n \pm \mu \sqrt{\alpha \beta} n.
\end{equation}
Given this set of edges $E(A, B)$, let $W(A, B)$ be the multiset of corresponding centres, that is, $w$ appears in $W(A, B)$ once for each edge $e \in E(A, B)$ in the induced subgraph of Cay$(G, S')$ on $C_w$. We claim that $|W(A, B)| = 2 |E(A, B)|$. To see this, write $e = (u, v)$ and note that if $e \subset C_{w_1}, C_{w_2}$, then $u = w_1 + s_1$, $v = w_1 + s_2$, $u = w_2 + s'_1$ and $v = w_2 + s'_2$ for $s_1, s_2, s'_1, s'_2 \in S$. This implies that $s_1 + s_2 = s'_1 + s'_2$, but since there are no non-trivial solutions to this equation, we must have $s'_1 = s_2$ and $s'_2 = s_1$. Therefore, $e$ is contained in at most two sets of the form $C_w$. To see that it is exactly two, note that $e \subset C_{u + s_1}, C_{u + s_2}$. 

The number of triples containing a pair from $E(A, B)$ and a vertex from $C$ is now the number of edges between $W := W(A, B)$ and $C$ in the graph Cay$(\ZZ_2^t, S)$, though we may have to remove as much as $2|W|$ to account for the fact that an edge between $W$ and $C$ will only give a true triple if the vertex in $C$ differs from the other two vertices of the triple. Therefore, by Lemma~\ref{lem:eml},
\begin{equation} \label{eqn:comexp2}
e(A, B, C) = e(W, C) \pm 2|W| = \frac{d}{n} |W||C| \pm \lambda \sqrt{\sum_{x \in W} w_x^2 |C|} \pm 2|W|.
\end{equation}
Since $w_x \leq \binom{d}{2}$ for all $x$, we have 
$$\sum_{x \in W} w_x^2 \leq \binom{d}{2} \sum_x w_x = \binom{d}{2} |W| \leq d^2 |E(A,B)|.$$ 
Substituting \eqref{eqn:comexp1} into \eqref{eqn:comexp2} using this fact yields the required result.
\end{proof}

We now prove our main theorem, Theorem~\ref{thm:main}, beginning with the triple expansion property. We will make use of the following result of de Caen~\cite{dC98} putting an upper bound on the sum of the squares of the degrees of a graph.

\begin{lemma} \label{lem:deC}
For any graph with $n$ vertices, $m$ edges and vertex degrees $d_1, d_2, \dots, d_n$, 
\[\sum_{i=1}^n d_i^2 \leq m \left(\frac{2m}{n-1} + (n-2)\right).\]
\end{lemma}

\begin{theorem} \label{thm:tripexp}
Suppose that $S \subseteq \ZZ_2^t$ contains no non-trivial solutions to the equation $s_1 + s_2 = s'_1 + s'_2$ and $\lambda(\textrm{Cay}(\ZZ_2^t, S)) \leq (1 - \epsilon)d$. Then, for all subsets $F$ of the skeleton $E$ of $H(\ZZ_2^t, S)$ with $|F| \leq |E|/2$, $|N_T(F)| \geq \frac{\epsilon^2}{64} d|F|$.
\end{theorem}

\begin{proof}
For each $x$, let $F_x = \{e \in F : e \subset C_x\}$. As in the proof of Theorem~\ref{lem:comb}, the assumption that $S$ contains no non-trivial solutions to the equation $s_1 + s_2 = s'_1 + s'_2$ implies that each edge in $F$ is contained in precisely two sets $C_x$. Therefore, $\sum_{x \in \ZZ_2^t} |F_x| = 2 |F|$. Suppose now that $X = \{x \in \ZZ_2^t : |F_x| \geq (1 - \delta) \binom{d}{2}\}$, where $\delta = \frac{\epsilon}{8}$. We claim that $\sum_{x \in X^c} |F_x| \geq \frac{\epsilon}{4}|F|$.

To prove the claim, we may clearly assume that $\sum_{x \in X} |F_x| \geq |F|$, for otherwise, 
$$\sum_{x \in X^c} |F_x| = 2|F| - \sum_{x \in X} |F_x| \geq |F|.$$  
Consider now the number of edges $N(X, X^c)$ in Cay$(G, S')$ between $X$ and its complement $X^c$. By Lemma~\ref{lem:eig} and the expander mixing lemma, this number is at least
\[\binom{d}{2}\frac{|X||X^c|}{n} - \frac{1}{2} (\lambda^2 - d) \frac{|X||X^c|}{n} = \frac{1}{2} (d^2 - \lambda^2) \frac{|X||X^c|}{n}\]
for some $\lambda \leq (1 - \epsilon)d$. Suppose now that $x \in X$ and $x + s_1 + s_2 \in X^c$. If $e = (x + s_1, x + s_2)$ is in $F$, we see that $e \in F_{x + s_1 + s_2}$, contributing one to $\sum_{x \in X^c} |F_x|$. If it is not in $F$, it contributes one to $\sum_{x \in X} |F_x^c|$. Therefore,
\[\sum_{x \in X^c} |F_x| + \sum_{x \in X} |F_x^c| \geq N(X, X^c)\]
and, since $|F_x^c| \leq \delta \binom{d}{2}$ for each $x \in X$,
\[\sum_{x \in X^c} |F_x| \geq N(X, X^c) - \delta \binom{d}{2}|X| \geq \frac{1}{2} (d^2 - \lambda^2) \frac{|X||X^c|}{n} - \delta \binom{d}{2}|X|.\]
Note now, since $|F| \leq \frac{1}{2} |E| = \frac{1}{4} \binom{d}{2} n$, that 
\[|X| \leq \frac{2|F|}{(1 - \delta) \binom{d}{2}} \leq  \frac{n}{2(1 - \delta)} \leq (1 + 2\delta)\frac{n}{2},\]
so $|X^c| \geq (1 - 2\delta)\frac{n}{2}$. Therefore, since $\lambda^2 \leq (1 - \epsilon) d^2$ and $\delta = \frac{\epsilon}{8}$,
\begin{align*}
\sum_{x \in X^c} |F_x| & \geq \frac{1}{2} (d^2 - \lambda^2) \frac{|X||X^c|}{n} - \delta \binom{d}{2}|X| \\
& \geq \frac{1}{4} \left((d^2 - \lambda^2) (1 - 2\delta) - 2 \delta d^2\right) |X| \\
& \geq \frac{1}{4} \left(\epsilon (1 - 2 \delta) - 2 \delta\right) d^2 |X| \geq \frac{\epsilon}{8} d^2 |X|.
\end{align*}
Since $\binom{d}{2}|X| \geq \sum_{x \in X}|F_x| \geq |F|$, the required claim, that $\sum_{x \in X^c} |F_x| \geq \frac{\epsilon}{4}|F|$, now follows.

Given $x \in X^c$, let $N_T(F_x)$ denote the set of triples $t \subset C_x$ such that $f \subset t$ for some $f \in F_x$ and $t \notin \Delta(F_x)$. Then
\[|N_T(F_x)| = \frac{1}{2} \sum_{y \in C_x} d(y) (d - 1 - d(y)),\]
where $d(y)$ is the degree of $y$ in the graph whose edges are $F_x$. This is because $N_T(F_x)$ includes every triple which contains an edge of both $F_x$ and $E_x \setminus F_x$ and the factor of $1/2$ accounts for the fact that we will include any admissible triple twice in this count. Now, by Lemma~\ref{lem:deC},
\[\sum_{y \in C_x} d^2(y) \leq |F_x| \left( \frac{2 |F_x|}{d-1} + d - 2\right),\]
so that, since $\sum_{y \in C_x} d(y) = 2|F_x|$,
\[|N_T(F_x)| \geq (d-1) |F_x| - \frac{|F_x|^2}{d-1} - \frac{1}{2} (d-2) |F_x| = \frac{d}{2} |F_x| - \frac{|F_x|^2}{d-1}.\]
Therefore, since $|F_x| \leq (1 - \delta) \binom{d}{2}$ for $x \in X^c$, $|N_T(F_x)| \geq \frac{\delta}{2} d |F_x|$. Since no triple appears in more than one $C_x$, it follows that
\[|N_T(F)| \geq \sum_{x \in X^c} |N_T(F_x)| \geq \frac{\delta}{2} d \sum_{x \in X^c} |F_x| \geq \frac{\epsilon^2}{64} d |F|,\]
as required.
\end{proof}

The edge expansion property from Theorem~\ref{thm:main} now follows as a simple corollary.

\begin{corollary} \label{cor:edgeexp}
Suppose that $S \subseteq \ZZ_2^t$ contains no non-trivial solutions to the equation $s_1 + s_2 = s'_1 + s'_2$ and $\lambda(\textrm{Cay}(\ZZ_2^t, S)) \leq (1 - \epsilon)d$. Then, for all subsets $F$ of the skeleton $E$ of $H(\ZZ_2^t, S)$ with $|F| \leq |E|/2$, $|N_E(F)| \geq \frac{\epsilon^2}{128} |F|$.
\end{corollary}

\begin{proof}
By Theorem~\ref{thm:tripexp}, there are at least $\frac{\epsilon^2}{64} d|F|$ triples which contain an edge from both $F$ and $E \setminus F$. Since each edge in $E \setminus F$ is contained in at most $2d$ of these triples, the result follows by division.
\end{proof}

We will prove Corollary~\ref{cor:mixing} on the rapid mixing of the random walk on the edges of $H(\ZZ_2^t, S)$ by constructing an auxiliary graph $G$ and then appealing to the following result~\cite[Theorem 3.3]{HLW06}.

\begin{lemma} \label{lem:mix}
Let $G$ be an $N$-vertex $D$-regular graph with $\lambda = \max_{i \neq 1} |\lambda_i(G)|$ and $\mathbf{p}_0$ a probability distribution on $V(G)$. Then the random walk on $G$ starting from the initial distribution $\mathbf{p}_0$ satisfies
\[\|\mathbf{p}_i - \mathbf{u}\|_2 \leq \left(\frac{\lambda}{D}\right)^i,\]
where $\mathbf{p}_i$ is the probability distribution on $V(G)$ after $i$ steps of the walk, $\mathbf{u}$ is the uniform distribution on $V(G)$ and $\|\mathbf{x}\|_2 = (\sum_i x_i^2)^{1/2}$ for $\mathbf{x} = (x_1, x_2, \dots, x_N) \in \mathbb{R}^N$.
\end{lemma}

To apply this lemma, we need to estimate $\lambda$. Recall that the \emph{edge expansion ratio} of a graph $G$ is 
\[h(G) = \min_{\{U : |U| \leq |V|/2\}} \frac{e(U, U^c)}{|U|}.\]
The following discrete analogue of the Cheeger inequality, due to Dodziuk~\cite{D84} and Alon and Milman~\cite{AM85}, places an upper bound on the second eigenvalue $\lambda_2$ of a graph $G$ in terms of its edge expansion ratio.

\begin{lemma} \label{lem:Cheeger}
If $G$ is an $N$-vertex $D$-regular graph with eigenvalues $\lambda_1 \geq \dots \geq \lambda_N$, then
\[\lambda_2 \leq D - \frac{h(G)^2}{2D}.\]
\end{lemma}

To estimate $\lambda_N$, we use a result of Desai and Roy~\cite{DR94}. To state their result, for any subset $U$ of the vertex set of a graph $G$, we define $b(U)$ to be the minimum number of edges that need to be removed from the induced subgraph $G[U]$ to make it bipartite.

\begin{lemma} \label{lem:DR}
If $G$ is an $N$-vertex $D$-regular graph with eigenvalues $\lambda_1 \geq \dots \geq \lambda_N$, then
\[\lambda_N \geq -D + \frac{\Psi^2}{4D},\]
where
\[\Psi = \min_{U \neq \emptyset} \frac{b(U) + e(U, U^c)}{|U|}.\]
\end{lemma}

{\bf Proof of Corollary~\ref{cor:mixing}:}
Consider the auxiliary graph $G$ whose vertices $V$ are the edges of the skeleton $E$ and where two vertices are joined if the union of the corresponding edges $e_1, e_2 \in E$ is in $T$. Note that $G$ is an $N$-vertex $D$-regular graph with $N = \frac{1}{2} \binom{d}{2} n$ and $D = 4d-8$. The random walk on $G$ is in one-to-one correspondence with the random walk on the original hypergraph $H(\ZZ_2^t, S)$ and the fact that $H$ is an $\frac{\epsilon^2}{64}d$-triple-expander easily implies that for any $U \subseteq V$ with $|U| \leq |V|/2$ the number of edges between $U$ and $U^c$ is at least $\frac{\epsilon^2}{64}d |U|$. Therefore, $h(G) \geq \frac{\epsilon^2}{64} d \geq \frac{\epsilon^2}{256}D$. By Lemma~\ref{lem:Cheeger}, this implies that 
\[\lambda_2(G) \leq D - \frac{h(G)^2}{2D} \leq \left(1 - \frac{\epsilon^4}{2^{17}}\right)D.\]
To estimate $\Psi$, we split into two cases. If $|U| < \frac{15}{16}N$, we use the fact that 
\[e(U, U^c) \geq \frac{\epsilon^2}{64}d \min(|U|, |U^c|) \geq \frac{\epsilon^2}{2^{10}} d |U|\] 
to conclude that $e(U, U^c)/|U| \geq \frac{\epsilon^2}{2^{10}} d$. On the other hand, if $|U| \geq \frac{15}{16} N$, the corresponding edge set $F$ in the skeleton $E$ of $H$ has at least $\frac{3}{4} \binom{d}{2}$ edges in at least $\frac{3}{4} n$ of the sets $C_x$. By supersaturation, there exists a constant $c > 0$ such that each $C_x$ with at least $\frac{3}{4} \binom{d}{2}$ edges has at least $cd^3$ triangles with all edges in $F$. As there are at least $\frac{3}{4} n$ sets $C_x$ with this property, $F$ contains at least $\frac{3}{4} c d^3 n$ triangles, which in turn implies that $G[U]$ contains at least $\frac{3}{4} c d^3 n$ triangles. Since $G$ is an edge-disjoint union of triangles, the number of edges which must be removed to make $G[U]$ bipartite is at least the number of triangles in $G[U]$. That is, $b(U)$ is at least $\frac{3}{4} c d^3 n$, so $b(U)/|U| \geq c d$. In either case, we see that $\Psi = \Omega(\epsilon^2 D)$ and, hence, by Lemma~\ref{lem:DR}, there is a positive constant $c$, which we may assume is at most $2^{-17}$, such that 
\[\lambda_N(G) \geq -D + \frac{\Psi^2}{4D} \geq (-1 + c \epsilon^4) D.\]
Putting everything together, we see that $\lambda = \max(|\lambda_2(G)|, |\lambda_N(G)|) \leq (1 - c \epsilon^4) D$. Therefore, applying Lemma~\ref{lem:mix},
\[\|\mathbf{p}_i - \mathbf{u}\|_2 \leq \left(\frac{\lambda}{D}\right)^i \leq (1 - c \epsilon^4)^i,\]
as required.
\qed

\section{Further remarks}

{\bf Generalised constructions.}

For simplicity, we have worked throughout with the group $\ZZ_2^t$. However, a similar construction works over any abelian group $G$. Indeed, given a subset $S$ of $G$, we can let $H(G, S)$ be the $3$-uniform hypergraph with vertex set $G$ and edge set consisting of all triples of the form $(x + s_1, x+ s_2, x+ s_3)$, where $x \in G$ and $s_1, s_2, s_3 \in S \cup (-S)$ with $s_i \neq \pm s_j$ for $i \neq j$.  Alternatively, $H$ is a $3$-uniform hypergraph on the same vertex set whose triples correspond to certain degenerate triangles in Cay$(G, S')$, where
\[S' = \{s_1 + s_2 : s_1, s_2 \in S \cup (-S), s_1 \neq \pm s_2\}.\]
It is worth noting that we omit edges of the form $(x, x + s + s)$, where they exist, since they will typically only be contained in one set of the form $C_x = \{x + s: s \in S \cup (-S)\}$. Over $\ZZ_2^t$, such edges do not exist, so this issue does not arise.

It is also possible to define a variant of our construction by using longer sums. For instance, given a subset $S$ of $\ZZ_2^t$, we may let
\[S' = \{s_1 + \dots + s_{2\ell} : s_i \in S, s_i \neq s_j\}\]
and take $H$ to again be a $3$-uniform hypergraph whose triples correspond to certain degenerate triangles in Cay$(\ZZ_2^t, S')$. However, this generalised construction seems to have few tangible benefits over the $\ell = 1$ case, so we have not pursued it further.

\vspace{3mm}

{\bf Vertex expansion.}

Given a subset $F$ of the skeleton $E$ of a hypergraph $H$, one may also define its \emph{vertex neighbourhood}
\[N_V(F) = \{v \in V(H) : v \cup f \in T \textrm{ for some } f \in F \textrm{ and } v \notin V(F)\},\]
where $V(F)$ is the set of vertices of $H$ which are contained in some edge of $F$. Assuming $H$ has $n$ vertices, we then let
\[h_V(H) = \min_{\{F: |V(F)| \leq n/2\}} \frac{|N_V(F)|}{|V(F)|}\]
and say that $H$ is an \emph{$\epsilon$-vertex-expander} if $h_V(H) \geq \epsilon$. One might now ask if our construction $H(\ZZ_2^t, S)$ has this vertex expansion property for a random choice of $S$. This problem seems surprisingly delicate and we were unable to decide in which direction the truth lies. A positive solution would be of substantial interest and is likely to facilitate applications to extremal combinatorics, such as to the determination of the size Ramsey number of tight paths (see~\cite{DLMR17, LW18} for the current status of this problem). It would also be of great interest to find alternative constructions, preferably of bounded degree, with this vertex expansion property. 

\vspace{3mm}

{\bf Coboundary and cosystolic expansion.}

The progress by Evra, Kaufman, Kazhdan and Lubotzky~\cite{EK17, KKL16} on constructing bounded-degree hypergraphs with the topological overlap property stems from a connection to another, more combinatorial, expansion property known as coboundary expansion~\cite{DK12, LM06}. We will not attempt to describe this property here, but suffice to say that coboundary expansion and a slightly weaker notion known as cosystolic expansion are both known to imply the topological overlap property~\cite{DKW16, G10}. The papers~\cite{EK17, KKL16} (and the related work in~\cite{LLR15, LM15}) then proceed by showing that the constructions under consideration are cosystolic expanders, from which the desired topological overlap property follows.

In assessing whether our construction gives cosystolic expanders, it is tempting to appeal to a criterion established by Evra and Kaufman~\cite{EK17}. In the $3$-uniform case, this roughly says that if a hypergraph $H$ has the property that the skeleton graph and the link of each vertex are good expander graphs, then $H$ is a cosystolic expander. Unfortunately, this result does not apply in our situation, since the links in our construction are not good expanders. Nevertheless, we are still willing to conjecture that our construction yields cosystolic, and perhaps even coboundary, expanders.

\vspace{3mm}

{\bf Higher uniformities.}

The construction given in this paper does not seem to generalise to higher uniformities. To see why, recall that, given a set $S$ and $x \in \ZZ_2^t$, we define $C_x = \{x + s : s \in S\}$. The principal reason our $3$-uniform construction goes through is that any edge $(x+s_1, x+s_2)$ in $C_x$ also appears in $C_{x+s_1+s_2}$ as $((x+s_1+s_2) + s_2, (x+s_1+s_2) + s_1)$. The natural analogue of this observation in the $4$-uniform case is to consider a triple $(x+s_1, x+s_2, x+s_3)$ in $C_x$ and to note that this triple can be rewritten as $((x+s_1+s_2 + s_3) + s_2 + s_3, (x+s_1+s_2 + s_3) + s_3 + s_1, (x+s_1+s_2 + s_3) + s_1 + s_2)$. Therefore, if $s_i + s_j$ is in $S$ for all $i \neq j$, we see that the triple $(x+s_1, x+s_2, x+s_3)$ is also in $C_{x + s_1 + s_2 + s_3}$. However, the requirement that $s + s'$ is in $S$ for any distinct $s, s' \in S$ is a very strong one, implying that $S$ contains every non-zero element in its span. Since $S$ needs to span all of $\ZZ_2^t$ for Cay$(\ZZ_2^t, S)$ to even be connected, it would need to contain all non-zero elements of $\ZZ_2^t$. The construction would then reduce to taking the complete $4$-uniform hypergraph on $\ZZ_2^t$. However, despite the failure of this particular mechanism, it remains a very interesting problem to find simple constructions of sparse expanders in higher uniformities.

\vspace{5mm}
\noindent
{\bf Note added.} Since the first version of this paper appeared, there has been significant progress on some of the questions posed above. For instance, Gundert and Luria have shown that our construction does not yield cosystolic expanders, though this does not necessarily preclude the possibility that it satisfies the topological expansion property. Moreover, despite the negative result in the last paragraph above, the construction was recently extended to higher uniformities in~\cite{CTZ18}.

\vspace{5mm}
\noindent
{\bf Acknowledgements.} The author gratefully acknowledges the support of the Simons Institute for the Theory of Computing during part of the period when this paper was written. The author is also indebted to Noga Alon, who brought the problem of constructing high-dimensional expanders to his attention, and to Rajko Nenadov, Jonathan Tidor and Yufei Zhao for several valuable discussions.

\end{document}